\documentclass{article}
\usepackage{amssymb}
\usepackage{amsfonts}
\usepackage{amsmath}
\usepackage{amsthm}
\usepackage{amsfonts}
\usepackage[all]{xy}

\usepackage{enumerate}
\usepackage{setspace}
\usepackage{comment}
\usepackage{graphicx}
\usepackage{float}
\usepackage{tikz}
\tikzstyle{vertex}=[circle, draw, inner sep=0pt, minimum size=6pt]
\newcommand{\vertex}{\node[vertex]}
\usepackage{calc}
\usepackage{pgfplots}
\usepackage{subfigure}
\usepackage[percent]{overpic}

\newtheorem{theorem}{Theorem}[section]

\newtheorem{lemma}[theorem]{Lemma}
\newtheorem{corollary}[theorem]{Corollary}
\newtheorem{observation}[theorem]{Observation}

\theoremstyle{definition}
\newtheorem{definition}{Definition}[section]

\theoremstyle{remark}

\let\originalleft\left
\let\originalright\right
\renewcommand{\left}{\mathopen{}\mathclose\bgroup\originalleft}
\renewcommand{\right}{\aftergroup\egroup\originalright}

\newcommand{\rel}{\mathrm{nRel}}
\newcommand{\crel}{\mathrm{Rel}}

\begin{document}

\title{The Shape of Node Reliability}
\author{Jason Brown and Lucas Mol} 
\date{}

\maketitle

\begin{abstract}
Given a graph $G$ whose edges are perfectly reliable and whose nodes each operate independently with probability $p\in[0,1],$ the \textit{node reliability} of $G$ is the probability that at least one node is operational and that the operational nodes can all communicate in the subgraph that they induce.  We study analytic properties of the node reliability on the interval $[0,1]$ including monotonicity, concavity, and fixed points.  Our results show a stark contrast between this model of network robustness and models that arise from coherent set systems (including all-terminal, two-terminal and K-terminal reliability).
\end{abstract}



\section{Introduction}

There are a number of models of probabilistic network reliability based on the premise that in a (finite, undirected) graph, the nodes are always operational, but edges are independently operational with probability $p$ (failure in such a system may be due to random failures of components in the links joining nodes).  The well known \textit{all-terminal reliability} asks for the probability that all vertices can communicate with one another (that is, at least a spanning tree is operational).  This model generalizes to $K$-\textit{terminal reliability}, which asks the probability that all vertices in some particular subset $K$ can communicate with one another (we call the vertices in $K$ the \textit{target nodes}, with the target nodes ranging from two particular vertices in the well-studied \textit{two-terminal reliability} to the entire vertex set for all-terminal reliability).  An excellent survey of these measures can be found in \cite{ColbournBook}.

These models of reliability fit under the umbrella of \textit{coherence}. Let $X$ be a finite ground set; a \textit{coherent set system} $\mathcal{S}$ on $X$ is a subset of $\mathcal{P}(X)$, the powerset of $X$, that satisfies the following conditions:
\begin{enumerate}[i)]
\item if $S_1\in \mathcal{S}$ and $S_1\subseteq S_2\subseteq X$ then $S_2\in \mathcal{S}$ (i.e.\ $\mathcal{S}$ is closed under taking supersets in $X$),
\item $\emptyset\not\in\mathcal{S},$ and
\item $X\in \mathcal{S}.$
\end{enumerate}
The \textit{order} of $\mathcal{S}$ is the cardinality of the ground set $X.$  We think of the elements of $X$ as components of a system that either operate or fail, hence we call the sets in $\mathcal{S}$ the \textit{operational states}.  Coherence is then the natural property that if we start with an operational state and make any number of failed components operational it can only improve matters (that is, will not result in a failed state).  Let $X$ have cardinality $n$ and suppose that each element of $X$ is independently operational with probability $p\in(0,1).$  The \textit{reliability} of coherent set system $\mathcal{S}$ on $X,$ denoted $\crel(\mathcal{S};p),$ is the probability that the set of operational elements of $X$ is in $\mathcal{S}$; that is,
\begin{align}
\crel(\mathcal{S};p) &= \sum_{S \in \mathcal{S}}p^{|S|}(1-p)^{n-|S|}\label{coherenceexpansion}\\
& = \sum_{i=0}^{n} N_{i}p^{i}(1-p)^{n-i}\label{coherenceNexpansion}
\end{align}
where $N_i$ is the number of operational states of order $i$ for each $i\in\{1,\hdots,n\}.$  There are obvious relevant coherent set systems underlying each of the network models introduced earlier, all on the edge set of the graph -- in general for $K$-terminal reliability, the operational states are those edge subsets that connect all vertices of $K.$  This collection of subsets is easily seen to be closed under taking supersets in the edge set, as adding edges cannot possibly disconnect an already connected graph.

Birnbaum, Esary, and Saunders achieved several significant results in \cite{S-shape} that describe the general shape of a coherent reliability polynomial (that is, the reliability of any coherent set system) on the interval $[0,1]$.   

\begin{itemize}
\item The reliability of any coherent set system $\mathcal{S}$ is strictly increasing on $(0,1)$ and satisfies $\crel(\mathcal{S};0)=0$ and $\crel(\mathcal{S};1)=1.$

\item The reliability of any coherent set system of order at least $2$ has at most one fixed point in $(0,1)$.

\item When written in the form of (\ref{coherenceNexpansion}), the reliability of any coherent set system $\mathcal{S}$ with $N_{1}=0$ and $N_{n-1}=n$ has exactly one fixed point $\hat{p}\in(0,1)$.  Further, $\crel(\mathcal{S};p)<p$ for $p\in(0,\hat{p})$ and $\crel(\mathcal{S};p)>p$ for $p\in(\hat{p},1).$
\end{itemize}

Note that the first result listed above implies that the all-terminal reliability of any graph with at least two vertices is strictly increasing on $(0,1)$ and passes through the points $(0,0)$ and $(1,1)$, and the second implies that the all-terminal reliability of any graph with at least two edges has at most one fixed point in $(0,1)$.  The conditions $N_1=0$ and $N_{n-1}=n$ of the third result listed above mean simply that the system fails whenever at most one component is operational, and that the system is operational whenever at most one component fails, respectively.  For all-terminal reliability, these conditions are satisfied if and only if the graph lies on at least $3$ vertices and is $2$-edge-connected (i.e.\ has no bridges), so this final result implies that the all-terminal reliability of any $2$-edge-connected graph on at least $3$ vertices is S-shaped.  As a final note, it is easily verified that the all-terminal reliability of of any graph on at least $3$ vertices that is not $2$-edge-connected lies below the identity function for $p\in(0,1).$  So the results of Birnbaum et al.\ give us a very good idea of the shape of any all-terminal reliability polynomial (or in general any coherent reliability polynomial) on the interval $[0,1].$

These findings led to the definition of an \textit{S-shaped} (or \textit{sigmoid shaped}) curve.  A typical S-shaped curve is shown in Figure \ref{SShapePic}.  In general, a function $f$ is called \textit{S-shaped} on $[0,1]$ if it satisfies:
\begin{itemize}
\item $f(0)=0$ and $f(1)=1$,
\item $f'(p)>0$ for $p\in(0,1)$,
\item $f(p)$ has a unique fixed point $\hat{p}\in(0,1)$, and
\item $f(p)<p$ for $p\in(0,\hat{p})$ and $f(p)>p$ for $p\in(\hat{p},1).$
\end{itemize}
With this new terminology, the final result of Birnbaum et al.\ listed above says that coherent reliability polynomials are S-shaped under fairly weak conditions.\\

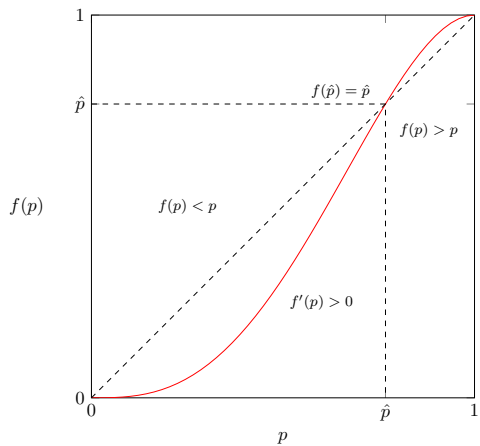
\begin{figure}
\begin{center}
\begin{tikzpicture}[scale=0.7]
\begin{axis}[
    axis equal image,
    scale only axis,
    x axis line style={-},
    y axis line style={-},
    xlabel={$p$},
    ylabel style={rotate=-90},
    ylabel={$f(p)$},
    ytick={0,0.7675918792,1},
    yticklabels={$0$, $\hat{p}$, $1$},
    xtick={0,0.7675918792, 1},
    xticklabels={$0$, $\hat{p}$, $1$},
    ymin=0,
    ymax=1,
    xmin=0,
    xmax=1,
    samples=50
]
    \draw[dashed] (axis cs:0,0.7675918792) -- (axis cs:0.7675918792,0.7675918792) -- (axis cs:0.7675918792,0);
   \addplot[domain=0:1,color=red] {x^4+4*x^3*(1-x)};
   \addplot[domain=0:1,dashed] {x};
   \node at (axis cs:0.25,0.5) {\footnotesize $f(p)<p$};
   \node at (axis cs:0.88,0.7) {\footnotesize $f(p)>p$};
   \node at (axis cs:0.65,0.8) {\footnotesize $f(\hat{p})=\hat{p}$};
   \node at (axis cs:0.6,0.25) {\footnotesize $f'(p)>0$};
\end{axis}
\end{tikzpicture}
\end{center}
\caption{A plot of an S-shaped function $f$.}
\label{SShapePic}
\end{figure}

Returning to the foundation of the network model, in some situations it is more realistic to assume that the edges are perfectly reliable and the nodes each operate independently with a given probability.  In particular, when the edges represent a wireless connection they may be nearly perfectly reliable.  We refer to this model as \textit{node reliability}, condensing the term \textit{residual node connectedness reliability} used in \cite{NodeRelComplexity1, Stivaros, NodeRelComplexity2}, for example.

\begin{definition}
Consider a network $G$ consisting of $n$ nodes each operating independently with probability $p\in[0,1]$.  The \textit{node reliability} of $G$, denoted $\rel(G;p)$, is the probability that at least one node is operational and that the operating nodes can all communicate in the induced subgraph that they generate.
\end{definition}

Node reliability can arise in a variety of applications. For example, (see  \cite{Stivaros}), a military may have many different missile sites, some with direct communication links between them (which we expect to be robust).  In the event that several missile sites are destroyed (each independently with probability $p$), we would like all of the remaining missile sites to be able to communicate with one another through the remaining network, so that targets for each surviving missile site can be chosen effectively.  As another example, consider a social network in which people are active (signed in, at work, etc.)\ independently with probability $p$.  The links in social networks (especially online social networks) are generally very reliable.  In case any urgent communication needs to occur through the network, we would like all of the active people to be able to communicate with one another.  

Like the other measures of reliability we have discussed, the node reliability of a graph is always a polynomial in $p$, as 
\begin{align} 
\rel(G;p) &= \sum_{C \in \mathcal{C}}p^{|C|}(1-p)^{n-|C|},\label{noderelexpansion}
\end{align}
where $\mathcal{C}$ is the collection of all vertex subsets that induce connected subgraphs of $G.$  We call these sets \textit{connected sets} and refer to $\mathcal{C}$ as the \textit{system of connected sets} of $G.$  As a simple example, the node reliability of the complete graph $K_n$ of order $n$ (i.e. the simple graph with $n$ nodes and all $\binom{n}{2}$ edges) is
\[
\rel(K_n;p)=1-(1-p)^{n},
\]
as any nonempty subset of vertices induces a connected subgraph (indeed, a complete subgraph).  The node reliability of $K_n$ is equivalent to the all-terminal reliability of a bundle of $n$ edges (the multigraph on two vertices with $n$ edges between them).

Much of the existing work on node reliability has concerned itself with finding optimal networks, should they exist, given constraints on the number of vertices and edges allowed (see \cite{NodeRel1,NodeRel2,Stivaros,NodeRel3}). Other research concerns the complexity of computing the polynomials; Sutner et al.\ showed in \cite{NodeRelComplexity2} that the problem of computing the node reliability polynomial is NP-hard, while Colbourn et al.\ \cite{NodeRelComplexity1} presented efficient algorithms for computing the node reliability polynomial of several restricted families of graphs.  Results on both of these problems for node reliability mirror those for $K$-terminal reliability to a large extent.

Despite the similarity of the formulations of node reliability and $K$-terminal reliability, and the similarity of discoveries on the two main problems (namely synthesis and computation issues), we demonstrate in this article that the shape of the node reliability polynomial on $(0,1)$ is very different from that of its coherent relatives.  In particular, we demonstrate the following key results:

\begin{itemize}
\item If a graph is sufficiently sparse, its node reliability polynomial has an interval of decrease in $(0,1)$.
\item The node reliability polynomial of any graph is concave down when $p$ is sufficiently close to $0.$
\item The node reliability polynomial of any tree of order at least $4$ has at least one inflection point in $(0,1).$
\item If a graph is sufficiently sparse and $2$-connected, its node reliability polynomial has at least two distinct inflection points in $(0,1).$
\item The node reliability polynomial of any sufficiently large graph with bounded maximum degree has at least two fixed points in $(0,1)$.
\end{itemize}

The first result listed above may come as a particular surprise.  The intuition is actually quite simple.  Note that the collection of connected sets of a graph is not a coherent set system in general; a connected set of a graph does not necessarily remain connected when we add vertices (this contrasts the situation for $K$-terminal reliability, where adding more edges cannot possibly disconnect an already connected graph).  In fact, it is easy to see that the connected set system of a graph $G$ is coherent if and only if $G$ is a complete graph (each singleton vertex set induces a connected graph, but any set made up of two non-adjacent vertices does not).  If we consider a sparse graph of order $n$, when $p$ is close to $\tfrac{1}{n}$, there is a fairly high probability that exactly one vertex is operating and thus a fairly high probability that the operational vertices induce a connected graph.  But as $p$ increases gradually from $\tfrac{1}{n}$, we are likely to have multiple operational nodes (but still not many), and since the graph is not dense it is unlikely that such a small set of nodes will induce a connected subgraph.  As an extreme example, if one considers disconnected graphs, we see that the node reliability of the complement of the complete graph of order $n$, $np(1-p)^{n-1}$, is increasing on $(0,1/n)$ and decreasing on $(1/n,1)$.  In the next section we prove that the node reliability of \textit{any} sufficiently sparse graph (connected or not) has an interval of decrease in $(0,1).$

\section{Monotonicity}\label{IncreaseDecrease}

It was proven in \cite{S-shape} that any coherent reliability polynomial is strictly increasing on $(0,1).$  We include our own short proof of this fact here as it is relevant to our work.  For any coherent set system $\mathcal{S}$ on a set $X$ of cardinality $n$, there is an associated set system $\mathcal{C_{S}}$ on $X$ given by
\[
\mathcal{C_S}=\{X-S\colon\ S\in \mathcal{S}\}.
\] 
The members of $\mathcal{C_S}$ are the sets of components whose failure does not render the graph inoperational.  Since $\mathcal{S}$ is coherent, the set system $\mathcal{C_S}$ is closed under \textit{containment}, making it a \textit{(simplicial) complex}.  We may write the reliability of the coherent system $\mathcal{S}$ in its \textit{$F$-form}
\[
\crel(\mathcal{S};p)=\sum_{k=0}^{n} F_k(1-p)^kp^{n-k},
\]
where $F_k$ is the number of sets of cardinality $k$ (in simplicial complex parlance, the number of \textit{faces} of cardinality $k$) in the complex $\mathcal{C_{S}}$. These coefficients satisfy \textit{Sperner's bounds} \cite{SpernerBound} for complexes:
\[
(k+1)F_{k+1}\leq (n-k)F_{k},
\]
for $k = 0,1,\ldots,n-1$. A straightforward computation yields
\begin{align}\label{Fform}
\crel'(\mathcal{S};p) = \sum_{k=0}^{n}\left[(n-k)F_k-(k+1)F_{k+1}\right](1-p)^kp^{n-k-1}.
\end{align}
Note that $F_0=1$ and $F_n=0$ follow immediately from the definition of coherent set system.  Let $t$ be the largest integer for which $F_t>0.$  The coefficient of the term corresponding to $k=t$ in (\ref{Fform}) is strictly positive as $F_{t}>1$ but $F_{t+1}=0,$ and the remaining coefficients are nonnegative by Sperner's bounds.  We conclude that the coherent reliability polynomial of any coherent set system is strictly increasing on $(0,1).$  As a corollary, the all-terminal reliability of a connected graph of order at least $2$ is strictly increasing.

While there might be an expectation that a similar result holds for node reliability, the issue of monotonicity is not so obvious for node reliability.  The coefficients of the $F$-form of the node reliability of a graph $G,$
\[
\rel(G;p)=\sum_{i=0}^n F_i(1-p)^ip^{n-i},
\]
fail to satisfy Sperner's bounds (the essential inequalities used above in the proof of monotonicity for coherent reliability polynomials) whenever $G$ is not complete, as then $F_{n-1}  = n$ and $F_{n-2} < n(n-1)/2$.  However, in spite of the failure of Sperner's bounds, there are non-complete graphs whose node reliability polynomials are always increasing on $(0,1).$  For example, consider the complete bipartite graph $K_{n,n}$ (the complete bipartite graph $K_{n_{1},n_{2}}$ consists of two cells of nodes, of sizes $n_{1}$ and $n_{2}$, respectively, with all $n_{1}n_{2}$ edges between the two cells).  The operational nodes induce a connected subgraph of $K_{n,n}$ if and only if at least one node from each bipartition set is operational, which occurs with probability $(1-(1-p)^n)^2$; or exactly one node is operational, which occurs with probability $2np(1-p)^{2n-1}$.  Thus, we have
\[ 
\rel(K_{n,n};p)=(1-(1-p)^{n})^2 + 2np(1-p)^{2n-1} = 1-2(1-p)^n+2np(1-p)^{2n-1}+(1-p)^{2n}.
\]
By a straightforward computation,
\[
\rel'(K_{n,n};p) =2n(1-p)^{n-1}\left[1-(2n-1)p(1-p)^{n-1}\right],
\]
and thus its sign is the same as that of $1-(2n-1)p(1-p)^{n-1}$ for all $p\in(0,1).$  Setting $f_n(p)=(2n-1)p(1-p)^{n-1}$, we see that $f_n'(p)=(2n-1)(1-p)^{n-2}[1-np]$, so that $p=\tfrac{1}{n}$ is the unique critical point of $f_n$ in $(0,1),$  and $f_n$ is maximized there.  We will demonstrate that
\[
f_n\left(\tfrac{1}{n}\right)=\tfrac{2n-1}{n}\left(1-\tfrac{1}{n}\right)^{n-1}<1
\] 
for all $n$.  Clearly $\tfrac{2n-1}{n}<2$, and $\left(1-\tfrac{1}{n}\right)^{n-1}<\tfrac{1}{2}$ as the sequence $a_n=\left(1-\tfrac{1}{n}\right)^{n-1}$ is decreasing for all $n\geq 2$ and $a_2=\tfrac{1}{2}.$  Hence we have
\[
f_n\left(\tfrac{1}{n}\right)<2\cdot \tfrac{1}{2}=1.
\] 
Therefore, $\rel'(K_{n,n};p)>0$ for all $p\in(0,1)$ for any $n\geq 2$, and we conclude that $\rel(K_{n,n};p)$ is increasing on $(0,1)$.

But we have already mentioned in the introduction that there are connected graphs whose node reliability polynomials are not always increasing on $(0,1)$ -- it does not even appear to be very rare for the node reliability polynomial to have an interval of decrease in $(0,1)$! Figure~\ref{path6noderel} shows a plot of the node reliability of a path of order $6$, and an interval of decrease between $0.2137$ and $0.5851$ is clearly evident. We have found that the node reliability polynomials of $37$ of the $112$ connected graphs of order $6$ have an interval of decrease, while the node reliability polynomials of $383$ of the $853$ connected graphs of order $7$ have an interval of decrease. 

\begin{figure}
\begin{center}
\begin{overpic}[scale=0.3]{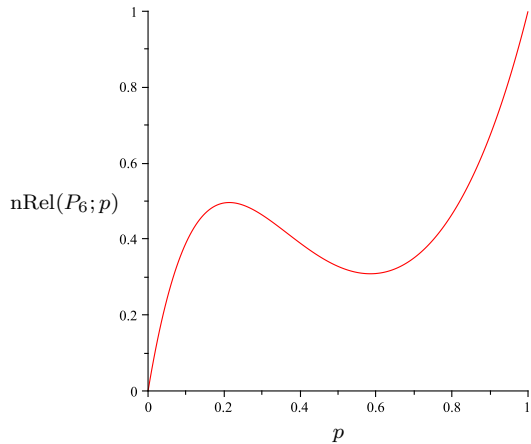}
\put(53,0){\footnotesize $p$}
\put(-20,52){\footnotesize $\rel(P_6;p)$}
\end{overpic}
\end{center}
\caption{Plot of the node reliability of a path of order $6$.}
\label{path6noderel}
\end{figure}

The expression for the node reliability given in (\ref{noderelexpansion}) gives rise to the convenient form
\begin{align}
\rel(G;p)=\sum_{k=1}^n c_kp^k(1-p)^{n-k},\label{cform}
\end{align}
where $c_k = c_k(G)$ is the number of connected sets of $G$ of order $k$ for each $k\in\{1,\hdots,n\}$. (Recall that a subset $C$ of vertices is called a connected set if and only if the induced subgraph $G[C]$ on $C$ is connected.)  We remark that while the problem of counting the number of connected sets in a graph has been studied in several different places in the literature  \cite{ConnectedSets,SubtreeMin,NodeRelComplexity2,SubtreeWang1, SubtreeWang2, SubtreeYan, SubtreeMinDeg, SubtreeDegreeSequence}, very little of this work distinguishes between connected sets of different orders, which node reliability inherently does.  

We refer to (\ref{cform}) as the \textit{$c$-form} of the node reliability, and we refer to the coefficients of the $c$-form collectively as the $c$-\textit{coefficients} of the node reliability polynomial.  The following straightforward observation giving the exact values of certain $c$-coefficients was made in \cite{Stivaros}.

\begin{observation}\label{ConnectedBasic}
Let $G$ be a connected graph of order $n$ and size $m$, let $\tau$ be the number of triangles of $G,$ and let $t$ be the number of cut vertices of $G.$ Then
\begin{enumerate}[\normalfont(i)]
\item $c_0=0$
\item $c_1=n,$
\item $c_2=m,$
\item $c_3=\left(\displaystyle\sum_{v\in V(G)}\tbinom{\mathrm{deg}(v)}{2}\right)-2\tau,$
\item $c_{n-1}=n-t,$ and
\item $c_n=1.$\hfill \qed
\end{enumerate}
\end{observation}

We will not require explicit formulas for any of the other coefficients; we will simply bound them in terms of these known coefficients.  We prove upper bounds on the $c$-coefficients of the node reliability polynomial in terms of lower $c$-coefficients.  These bounds are similar in spirit to Sperner's bounds, which were used in the proof that any coherent reliability polynomial is increasing on $(0,1).$

\begin{lemma}\label{GeneralBound}
For any graph $G$, 
\[
2c_k\leq (n-k+1)c_{k-1}
\]
for all $k\in\{2,\hdots,n\}$.  More generally,
\[
(k-t+1)c_k\leq \binom{n-t}{k-t}c_t
\]
for all $k\in\{2,\hdots,n\}$ and $t\in\{1,\hdots,k-1\}.$
\end{lemma}

\begin{proof}
Let $G$ be a connected graph on $n\geq 1$ vertices.  We first prove that for each $k\in\{1,\hdots,n\},$
\[
c_k(G)\geq n-k+1.
\]
We proceed by induction on $n.$  For the base case, when $n=1$ we have $c_1=1\geq 1-1+1,$ and the statement is verified.  Now suppose that for some $n\geq 2$, any connected graph $H$ of order $n-1$ satisfies $c_k(H)\geq n-k$.   Let $G$ be a connected graph of order $n.$  Let $v$ be a vertex whose removal does not disconnect $G.$  There are exactly $c_k(G-v)$ connected sets of $G$ of order $k$ that do not contain $v,$ and there must be at least one connected set of $G$ of order $k$ containing $v$  (of course we may successively add vertices to $\{v\}$ so that the set is connected at every step, as the graph is connected).  Thus we have
\[
c_k(G)\geq c_k(G-v)+1.
\]
Now by the induction hypothesis applied to $G-v,$
\[
c_k(G)\geq c_k(G-v)+1\geq n-k+1.
\]

Now we are ready to prove the statement of the Lemma.  For each $k\in\{1,\hdots, n\},$ let $C_k$ be the collection of connected sets of $G$ of order $k.$  For any $k\geq 2,$ consider a member $S$ of $C_k.$  The induced subgraph $G[S]$ contains at least $k-t+1$ connected sets of order $t$ by the argument in the previous paragraph.  Clearly, any connected set of $G[S]$ must also be a connected set of $G.$  Therefore, every member of $C_k$ can be written in the form $W\cup X$ for some $W\in C_{t}$ and $X\subseteq V(G)\backslash W$ in at least $k-t+1$ distinct ways.  The total number of pairs $(W,X)$ where $W\in C_{t}$ and $X\subseteq V(G)\backslash W$ is
\[
\binom{n-t}{k-t}c_{t}.
\]
Since each member of $C_k$ arises from at least $k-t+1$ of these pairs, we have
\[
(k-t+1)c_k\leq \binom{n-t}{k-t}c_t\qedhere
\]
\end{proof}

Now we are ready to prove the main result of this section.  

\begin{theorem}\label{Decrease}
If $G$ is a graph of order $n$ and size $m\leq 0.0851n^2,$ then $\rel(G;p)$ has an interval of decrease in $(0,1)$.
\end{theorem}

\begin{proof}
Let $G$ be as in the theorem statement.  A straightforward computation gives
\begin{align}\label{RelDerivForm}
\rel'(G;p)=\sum_{k=1}^np^{k-1}(1-p)^{n-k}\left[kc_k-(n-k+1)c_{k-1}\right],
\end{align}
where $c_k$ is the number of connected sets of $G$ of order $k$ for $k\in\{0,\hdots,n\}$ (recall that $c_0=0$).  We find directly using the facts that $c_1=n$ and $c_2=m$ from Observation \ref{ConnectedBasic} that the sum of the first two terms (corresponding to $k=1$ and $k=2$) of the sum in (\ref{RelDerivForm}) is given by: 
\begin{align*}
(1-p)^{n-2}\left[n(1-np)+2mp\right]
\end{align*}
We now bound the remaining terms in the sum from (\ref{RelDerivForm}) for any $p\in(0,1).$  For ease of reading we let
\begin{align}\label{sum1}
S=\sum_{k=3}^np^{k-1}\left(1-p\right)^{n-k}\left[kc_k-(n-k+1)c_{k-1}\right],
\end{align}
so that 
\begin{align}\label{Eval}
\rel'(G;p)=(1-p)^{n-2}\left[n(1-np)+2mp\right]+S.
\end{align}

We claim that $S\leq m\left[p-\tfrac{1}{n-1}+\tfrac{(1-p)^{n-1}}{n-1}\right]$.  We first use the fact that
\[
(n-k+1)c_{k-1}\geq 2c_k
\]
for all $k\in\{2,\hdots,n\}$ by Lemma \ref{GeneralBound}.  This means that
\[
kc_k-(n-k+1)c_{k-1}\leq (k-2)c_k,
\]
so that from (\ref{sum1}) we obtain
\begin{align}\label{sum2}
S\leq \sum_{k=3}^np^{k-1}\left(1-p\right)^{n-k}(k-2)c_k.
\end{align}
Using Lemma \ref{GeneralBound} again for $t=2,$ we have
\[
(k-1)c_k\leq \binom{n-2}{k-2}c_2=m\binom{n-2}{k-2}
\]
for all $k\in\{3,\hdots, n\}.$  From (\ref{sum2}), we obtain
\begin{align}\label{sum3}
S\leq m\cdot\sum_{k=3}^np^{k-1}\left(1-p\right)^{n-k}\left(\frac{k-2}{k-1}\right)\binom{n-2}{k-2}.
\end{align}
The sum in (\ref{sum3}) can be evaluated using binomial identities to obtain
\begin{align}
S\leq m\left[p-\tfrac{1}{n-1}+\tfrac{(1-p)^{n-1}}{n-1}\right], \label{BoundOnS}
\end{align}
as claimed.

Substituting the upper bound for $S$ from (\ref{BoundOnS}) into (\ref{Eval}), we get
\begin{align}
\rel'(G;p)&\leq \left(1-p\right)^{n-2}(n(1-np)+2mp)+m\left[p-\tfrac{1}{n-1}+\tfrac{(1-p)^{n-1}}{n-1}\right]\\
&=(1-p)^{n-2}\left[n(1-np)+m(2p+\tfrac{1-p}{n-1})\right]+m\left[p-\tfrac{1}{n-1}\right].\label{sum4}
\end{align}
We substitute $p=\tfrac{r}{n}$ into (\ref{sum4}) for some fixed $r \in (1,2)$; the exact value of $r$ will be determined shortly.
\begin{align*}
\rel'\left(G;\tfrac{r}{n}\right)&\leq\left(1-\tfrac{r}{n}\right)^{n-2}\left[n(1-r)+m\left(\tfrac{2r}{n}+\tfrac{n-r}{n(n-1)}\right)\right]+m\left[\tfrac{r}{n}-\tfrac{1}{n-1}\right]\\
&<\left(1-\tfrac{r}{n}\right)^{n-2}\left[n(1-r)+\tfrac{m}{n}\left(2r+1\right)\right]+\tfrac{m}{n}(r-1)\\
&=\left(1-\tfrac{r}{n}\right)^{n-2}\left[n(1-r)+\tfrac{m}{n}\left(2r+1\right)+\tfrac{m}{n}(r-1)\left(1-\tfrac{r}{n}\right)^{-(n-2)}\right]
\end{align*}
We now show that $\left( 1-\tfrac{r}{n}\right) ^{-(n-2)}<e^r$. We set $f(x) = \left( 1-\tfrac{r}{x}\right)^{-(x-2)}$ for $x \in (r,\infty)$. Now, using the fact that $-\ln(1-y) > y$ for $y \in (0,1)$, we see that
\begin{align*} 
f^{\prime}(x) & = \left( 1-\tfrac{r}{x}\right) ^{-(x-2)} \left( -\ln \left( 1 - \frac{r}{x} \right) + \frac{r(2-x)}{x^{2} \left( 1 - \frac{r}{x} \right)} \right)\\
 & > \left( 1-\tfrac{r}{x}\right) ^{-(x-2)} \left( \frac{r}{x} + \frac{r(2-x)}{x^{2} \left( 1 - \frac{r}{x} \right)} \right)\\
 & = \left( 1-\tfrac{r}{x}\right) ^{-(x-2)}\left(\frac{r(2-r)}{x(x-r)}\right)\\
 & > 0
\end{align*}
since $r<2$ and $x>r$. Thus $\left( 1-\tfrac{r}{n}\right) ^{-(n-2)}$ is increasing, and as $\displaystyle{\lim a_n=e^r}$, we have that $\left( 1-\tfrac{r}{n}\right) ^{-(n-2)} < e^{r}$.   
Thus, 
\[
\rel'(G;\tfrac{r}{n})<\left(1-\tfrac{r}{n}\right)^{n-2}\left[n(1-r)+\tfrac{m}{n}\left(2r+1+(r-1)e^r\right)\right].
\]
We find that if 
\begin{align}\label{LastIneq}
n(1-r)+\tfrac{m}{n}\left(2r+1+(r-1)e^r\right)\leq 0
\end{align}
then $\rel'(G;\tfrac{r}{n})<0$ and $\rel(G;p)$ is decreasing at $p=\tfrac{r}{n}.$  We rearrange (\ref{LastIneq}) to obtain the sufficient condition
\begin{align*}\label{Final}
m\leq \left[\tfrac{r-1}{2r+1+(r-1)e^r}\right]n^2.
\end{align*}
Using a computer algebra system, we determine that the function 
\[
f(r)=\tfrac{r-1}{2r+1+(r-1)e^r}
\]
reaches a maximum on $(1,\infty)$ of 
\[
f(\hat{r})\approx 0.08510464442
\] 
where $\hat{r}\approx 1.729474372$ (one can solve for $\hat{r}$ exactly in terms of the well known Lambert $W$ function).  We conclude that $\rel(G;p)$ has an interval of decrease in $(0,1).$
\end{proof}

One might ask how dense a graph needs to be to ensure that its node reliability polynomial is increasing on the entire interval $(0,1)$, and indeed, it must be very dense. Consider the graph formed from the complete graph $K_{n-1}$ by adding a single pendant edge.  Let us denote this graph by $K_{n-1} \circ K_{2}$ (the \textit{bonding} of $K_{n-1}$ and $K_{2}$ at a vertex).  Note that $K_{n-1}\circ K_2$ has $n$ vertices and only $n-2$ nonedges.  The reader can verify that
\begin{align*}
\rel(K_{n-1} \circ K_{2};p)=1-p(1-p)+p(1-p)^{n-1}-(1-p)^n
\end{align*}
for all $n\geq 2.$  We find that
\begin{align}\label{CompletePlusLeafDer}
\rel'\left(K_{n-1}\circ K_2;p\right)= 2p-1+(1-p)^{n-2}(n-2np+1), 
\end{align}
and evaluating at $p=\tfrac{2}{5}$ gives
\[
-\tfrac{1}{5}+\left(\tfrac{3}{5}\right)^{n-2}\left(\tfrac{1}{5}n+1\right).
\]
We set $g(x) = -\tfrac{1}{5}+\left(\tfrac{3}{5}\right)^{x-2}\left(\tfrac{1}{5}x+1\right)$ and find that
\[
g^{\prime}(x)=\left(\tfrac{3}{5}\right)^{x-2}\left[\ln\left(\tfrac{3}{5}\right)\left(\tfrac{1}{5}x+1\right)+\tfrac{1}{5}\right],
\]
which is negative for all $x\geq 0.$  Further, we find that $g(7)<0,$ so that $g(n)<0$ for all $n\geq 7.$  Since $g(n)=\rel'\left(K_{n-1}\circ K_2;\tfrac{2}{5}\right)$ for all integers $n\geq 2$ we conclude that $\rel(K_{n-1} \circ K_{2};p)$ is decreasing at $p =\tfrac{2}{5}$ for all $n \geq 7$.  Figure \ref{CompletePlusLeafPlot} shows a plot of $\rel(K_9\circ K_2;p)$ which has a clearly evident interval of decrease.

\begin{figure}
\centering
\begin{overpic}[scale=0.3]{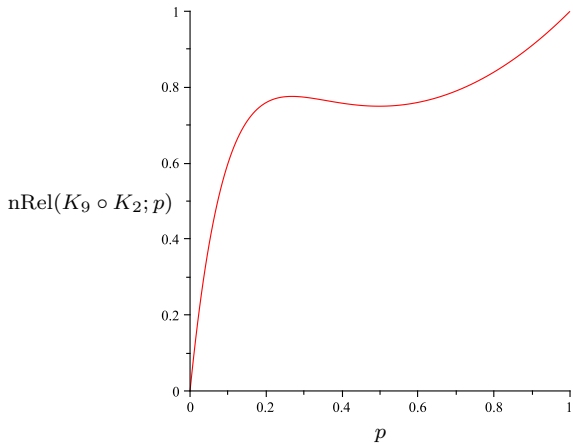}
\put(53,0){\footnotesize $p$}
\put(-30,52){\footnotesize $\rel(K_9\circ K_2;p)$}
\end{overpic}
\caption{Plot of the node reliability of $K_9\circ K_2$}
\label{CompletePlusLeafPlot}
\end{figure}

\section{Concavity and inflection points}\label{Concavity}

We now turn to the question of concavity and points of inflection.  Points of inflection of reliability polynomials are of interest as they indicate an important change in the behaviour of the graph.  The return on a small increase in the reliability of the individual components changes from decreasing to increasing, or the other way around.  This can give some indication as to whether or not it is worth the extra cost or effort involved to make each individual component a little bit more reliable.

We first summarize work done on the concavity and inflection points of all-terminal reliability polynomials.  In \cite{BrownRel}, Brown, Ko\c{c} and Kooij proved that the all-terminal reliability of almost every simple graph has an inflection point in $(0,1)$. The arguments there can be extended easily to show that for any coherent set system $\mathcal{S}$ with $N_1 = 0$ and $N_{n-1} = n,$ the coherent reliability polynomial $\crel(\mathcal{S};p)$ is concave up for $p>0$ sufficiently close to $0$ and concave down for $p<1$ sufficiently close to $1$.  Hence, under these weak conditions, a coherent reliability polynomial has at least one point of inflection in $(0,1)$.  In \cite{Graves}, Graves demonstrated that coherent reliability polynomials can have two inflection points in $(0,1)$.  Later, several families of all-terminal reliability polynomials having two inflection points in $(0,1)$ were presented in \cite{BrownRel}.  Finally, the fact that all-terminal reliability polynomials can have arbitrarily many inflection points in the interval $(0,1)$ was proven for multigraphs \cite{ManyInflectionPoints}.  For all-terminal reliability, or more generally for coherent reliability, the families which are known to have more than one point of inflection in $(0,1)$ are rather thin -- very few examples of any particular order $n$ are known.

What is the case for node reliability?  It is not difficult to see that for any $n\geq 2$ the complete graph on $n$ vertices is concave down on the entire interval $(0,1),$ as 
\[
\rel(K_n;p)=1-(1-p)^n,
\]
so that
\[
\rel''(K_n;p)=-n(n-1)(1-p)^{n-2}<0
\]
for $n\geq 2$ and $p\in (0,1).$  We conclude that the node reliability polynomial of a complete graph has no inflection points in the interval $(0,1).$  The remainder of this section concerns finding node reliability polynomials with one or more inflection point.

By a straightforward computation, the second derivative of the node reliability of $G$ is given by 
\begin{align}
\rel''(G;p)=\sum_{k=1}^{n-1}d_kp^{k-1}(1-p)^{n-k-1}, \label{RelSecondDeriv}
\end{align}
with
\begin{align}
d_k=(k+1)kc_{k+1}-2k(n-k)c_k +(n-k+1)(n-k)c_{k-1}, \label{dcoeff}
\end{align}
where $c_k$ is the number of connected sets of $G$ on $k$ vertices for $k\in\{1,\hdots,n\}.$  We use this notation for $d_k$ throughout the remainder of this article.

\begin{lemma}\label{ConcaveDown}
Let $G$ be a graph on $n\geq 2$ vertices.  The node reliability of $G$ is concave down for $p$ sufficiently close to $0$.
\end{lemma}

\begin{proof}
Consider $\rel''(G;p)$ in the form given in (\ref{RelSecondDeriv}).  Notice that for $k\geq 2$ we have
\[
\lim_{p\rightarrow 0}d_kp^{k-1}(1-p)^{n-k-1}=0,
\]
while the term corresponding to $k=1$ in (\ref{RelSecondDeriv}) has limit
\[
\lim_{p\rightarrow 0}d_1(1-p)^{n-2}=d_1.
\]
Thus when $p$ is sufficiently close to $0$, the sign of $\rel''(G;p)$ will be the same as the sign of $d_1$ (as long as this value is nonzero).  From (\ref{dcoeff}), $d_1$ is given by
\[
d_1=2c_2-2(n-1)c_1+n(n-1)c_0,
\]
By Observation \ref{ConnectedBasic}, we have
\[
d_1=2m-2n(n-1)+0\leq 2\binom{n}{2}-2n(n-1)=-n(n-1).
\]
Therefore, $\rel''(G;p)<0$ and thus $\rel(G;p)$ is concave down for $p$ sufficiently close to $0$.
\end{proof}

Lemma \ref{ConcaveDown} demonstrates another major difference between the shape of the node reliability polynomial and the shape of the all-terminal reliability polynomial -- while the node reliability of any graph on $n\geq 2$ vertices is concave down near $p=0$ by Lemma \ref{ConcaveDown}, the all-terminal reliabiliy of any graph on $n\geq 3$ vertices is concave up near $p=0$ (see \cite{BrownRel}, Proposition 2).

At the opposite end of the interval $[0,1],$ when $p$ is close to $1$, the node reliability may be concave up or concave down.  In the next theorem we prove that the node reliability of any tree is concave up near $p=1$ in order to reach the conclusion that it has at least one inflection point in $(0,1).$  This is again very different from the case for $K$-terminal reliability (including two-terminal and all-terminal reliability); the $K$-terminal reliability of a tree $T$ is equal to $p^k$ (with $k$ being the number of edges in a minimum subtree containing all vertices of $K$), and hence has no inflection points in $(0,1).$

\begin{theorem}\label{TreeInflection}
Let $T$ be a tree on $n\geq 4$ vertices.  The node reliability polynomial of $T$ has at least one point of inflection in $(0,1).$
\end{theorem}

\begin{proof}
First suppose that $T\cong K_{1,n-1}$ for some $n\geq 4.$  We find that
\[
\rel(K_{1,n-1};p)=p+(n-1)p(1-p)^{n-1},
\]
so straightforward computation gives
\[
\rel''(K_{1,n-1};p)=(n-1)^2(np-2)(1-p)^{n-3}.
\]
A simple analysis demonstrates that $\rel(K_{1,n-1};p)$ is concave down on $\left(0,\tfrac{2}{n}\right)$ and concave up on $\left(\tfrac{2}{n},1\right),$ so that the intended conclusion holds.

Now let $T$ be a tree on $n\geq 4$ vertices that is not isomorphic to $K_{1,n-1}.$  Consider $\rel''(T;p)$ in the form given in (\ref{RelSecondDeriv}).  Note that 
\[
\lim_{p\rightarrow 1}\rel''(T;p)=d_{n-1},
\]
as all terms apart from the $k=n-1$ term approach $0$ as $p\rightarrow 1$.  Thus, when $p$ is sufficiently close to $1$, the sign of $\rel''(T;p)$ will be the same as the sign of $d_{n-1}$ (as long as $d_{n-1}$ is nonzero).  From (\ref{dcoeff}),
\begin{align}\label{LeadingCoeff1}
d_{n-1}=n(n-1)c_n-2(n-1)c_{n-1}+2c_{n-2}.
\end{align}
By Observation \ref{ConnectedBasic}, $c_n=1$ and $c_{n-1}=n-t$, where $t$ is the number of cut vertices of $T$.  Since $T$ is a tree, we can write $n-t=r$ where $r$ is the number of leaves of $T,$ so that $c_{n-1}=r.$    Further, $c_{n-2}=\binom{r}{2}+s,$ where $s$ is the number of leaves adjacent to a vertex of degree $2$ in $G$, as the connected sets of order $n-2$ consist of either all vertices but a pair of leaves or all vertices but a leaf and an adjacent vertex of degree $2$.  Substituting these values into (\ref{LeadingCoeff1}), we obtain
\begin{align*}
d_{n-1}&=n(n-1)-2(n-1)r+2\left[\binom{r}{2}+s\right]\\
&=n(n-1)-2(n-1)r+r(r-1)+2s\\
&=2s+n(n-1)-r(2n-r-1).
\end{align*}
Since $T\not\cong K_{1,n-1},$ we have $r<n-1$ so that
\[
r(2n-r-1)=n(n-1)-(n-r)(n-1-r)<n(n-1),
\]
and finally
\[
d_{n-1}=2s+n(n-1)-r(2n-r-1)>2s.
\]
We conclude that $\rel''(T;p)$ is positive for $p$ sufficiently close to $1$, and therefore that $\rel(T;p)$ is concave up for $p$ sufficiently close to $1$.  Recall from Lemma \ref{ConcaveDown} that $\rel(T;p)$ is concave down for $p$ sufficiently close to $0.$  We conclude that $\rel(T;p)$ has at least one inflection point in $(0,1).$
\end{proof}

In the proof of Theorem \ref{TreeInflection} we saw that the node reliability polynomial of the star $K_{1,n-1}$ has exactly one inflection point in $(0,1)$ for any $n\geq 4,$ and we conjecture that all trees of order at least $4$ have exactly one inflection point in $(0,1).$  While many of the `S-shaped' all-terminal reliability polynomials also appear to have a single point of inflection in $(0,1),$ the node reliability of any tree on $n\geq 4$ vertices appears to have an `N-shape' on $(0,1)$ as opposed to the `S-shape' of the all-terminal reliability polynomials.  Figure \ref{TreeShape} provides a plot showing the node reliability polynomials of all trees on $7$ vertices.

\begin{figure}
\begin{center}
\begin{overpic}[scale=0.3]{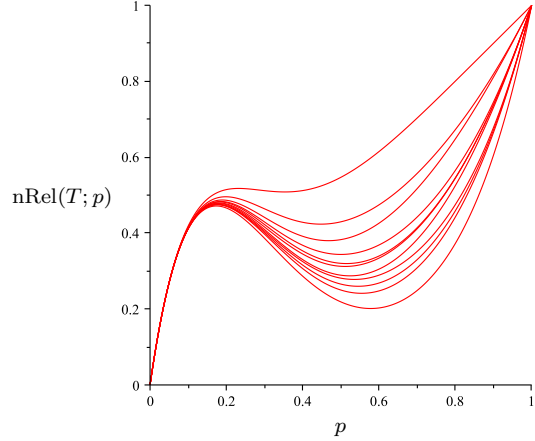}
\put(53,0){\footnotesize $p$}
\put(-20,52){\footnotesize $\rel(T;p)$}
\end{overpic}
\end{center}
\caption{Node reliability polynomials of all trees on $7$ vertices.}
\label{TreeShape}
\end{figure}

We next present a family of graphs whose node reliability polynomials each have at least two inflection points in $(0,1).$  Unlike the examples for coherent and all-terminal reliability polynomials, our family provides numerous examples of each order $n.$  We will require the following lemma.

\begin{lemma}\label{ConcaveDownNear1}
Let $G$ be a $2$-connected graph.  The node reliability of $G$ is concave down for $p\in (0,1)$ sufficiently close to $1$.
\end{lemma}

\begin{proof}
Again we consider $\rel''(G;p)$ in the form given in (\ref{RelSecondDeriv}).  Let $\ell$ be the order of a smallest vertex cut-set in $G$ (note that $\ell\geq 2$ as $G$ is $2$-connected by assumption).  We must have
\[
c_k=\binom{n}{k} \ \mbox{for all} \ k> n-\ell
\]
and
\[
c_{n-\ell}< \binom{n}{\ell}.
\]
Substituting these values into (\ref{dcoeff}), for any $k>n-\ell+1$ we have
\begin{align*}
d_k&=(k+1)k\binom{n}{k+1}-2k(n-k)\binom{n}{k}+(n-k+1)(n-k)\binom{n}{k-1}\\
&=\frac{n!}{(k-1)!(n-k-1)!}-2\frac{n!}{(k-1)!(n-k-1)!}+\frac{n!}{(k-1)!(n-k-1)!}\\
&=0,
\end{align*}
while for $k=n-\ell+1$ we have
\begin{align*}
d_{n-\ell+1}&=(n-\ell+2)(n-\ell+1)c_{n-\ell+2}\\
& \hspace{1cm} -2(n-\ell+1)(\ell-1)c_{n-\ell+1}+\ell(\ell-1)c_{n-\ell}\\
&<(n-\ell+2)(n-\ell+1)\binom{n}{n-\ell+2}\\
& \hspace{1cm} -2(n-\ell+1)(\ell-1)\binom{n}{n-\ell+1}+\ell(\ell-1)\binom{n}{n-\ell}\\
&=\frac{n!}{(n-\ell)!(\ell-2)!}-2\frac{n!}{(n-\ell)!(\ell-2)!}+\frac{n!}{(n-\ell)!(\ell-2)!}\\
&=0.
\end{align*}
So $d_k=0$ for $k>n-\ell+1$ and $d_{n-\ell+1}<0.$  Substituting $d_k=0$ for $k>n-\ell+1$ into (\ref{RelSecondDeriv}), we obtain
\begin{align*}
\rel''(G;p)&=\sum_{k=1}^{n-\ell+1}d_kp^{k-1}(1-p)^{n-k-1}\\
&=(1-p)^{\ell-2}\sum_{k=1}^{n-\ell+1}d_kp^{k-1}(1-p)^{n-k-\ell+1}
\end{align*}
Clearly we have $\displaystyle\lim_{p\rightarrow 1^-}\rel''(G;p)=0.$  However, we can determine the sign of $\rel''(G;p)$ for $p\in(0,1)$ sufficiently close to $1$ by considering the sign of the sum
\[
\sum_{k=1}^{n-\ell+1}d_kp^{k-1}(1-p)^{n-k-\ell+1}.
\] 
We have
\[
\lim_{p\rightarrow 1^-}\sum_{k=1}^{n-\ell+1}d_kp^{k-1}(1-p)^{n-k-\ell+1}=d_{n-\ell+1},
\]
as all terms in the sum apart from the $k=n-\ell+1$ term have a positive power of $(1-p).$  Since $d_{n-\ell+1}<0$, we conclude that for $p\in(0,1)$ sufficiently close to $1$  we have $\rel''(G;p)<0$.
\end{proof}

Now we are ready to show that there are graphs whose node reliability polynomials have two or more inflection points in $(0,1).$

\begin{theorem}\label{TwoPOI}
Let $G$ be a graph of order $n$ and size $m$.  If $m\leq 0.0851n^2$ and $G$ is $2$-connected then $\rel(G;p)$ has at least two distinct points of inflection in $(0,1).$
\end{theorem}

\begin{proof}
By Theorem \ref{Decrease}, $\rel(G;p)$ contains an interval of decrease in $(0,1)$.  In fact, $\rel'\left(G;\tfrac{\hat{r}}{n}\right)<0,$ where $\hat{r}\approx 1.729474372$, as discussed in the proof of Theorem \ref{Decrease}.  Since $\rel(G;p)<1$ for all $p\in (0,1)$ and $\rel(G;1)=1$, $\rel(G;p)$ must be increasing on some neighbourhood $(\overline{p},1).$  Let $\hat{p}\in(\overline{p},1)$ so that $\hat{p}>\tfrac{\hat{r}}{n}$ and $\rel'\left(G;\hat{p}\right)>0.$  By the Mean Value Theorem, there is some point $c\in\left(\tfrac{\hat{r}}{n},\hat{p}\right)$ such that
\[
\rel''(G;c)=\frac{\rel'(G;\hat{p})-\rel'\left(G;\tfrac{\hat{r}}{n}\right)}{\hat{p}-\tfrac{\hat{r}}{n}}>0.
\]
Therefore, $\rel(G;p)$ is concave up at some point $c$ inside the interval.  By Lemmas \ref{ConcaveDown} and \ref{ConcaveDownNear1}, $\rel(G;p)$ is concave down for all $p\in(0,1)$ sufficiently close to $0$ and concave down for all $p\in(0,1)$ sufficiently close to $1$, so we conclude that $\rel(G;p)$ has at least two points of inflection in $(0,1).$
\end{proof}

Thus it is not so rare for node reliabiliy polynomials to have two or more points of inflection in $(0,1).$  By comparison, the families of graphs presented in \cite{BrownRel, ManyInflectionPoints} whose all-terminal reliability polynomials have two or more points of inflection in $(0,1)$ contain far fewer graphs of each order $n.$

\section{Fixed points}\label{FixedPoints}

A key result proven in \cite{S-shape} is that the reliability polynomial of any coherent set system of order at least $2$ has at most one fixed point in $(0,1).$  As a corollary, the all-terminal reliability of any connected graph with at least $2$ edges has at most one fixed point in $(0,1),$ and the two-terminal reliability of any connected graph has at most one fixed point in $(0,1)$ as long as the target nodes are non-adjacent.  The fixed point $\hat{p}$ for the two-terminal reliability of a network plays a particularly interesting role when we \textit{iterate} the network structure (that is, when we replace each edge $e$ in the network with a copy of the network itself, identifying the target nodes with the endpoints of $e$).  The two-terminal reliability of this iterated network is found by composing the two-terminal reliability polynomial with itself.  If we do this repeatedly, then for $p>\hat{p}$, the two-terminal reliability of the iterated structure tends to $1$, while for $p<\hat{p},$ the reliability tends to $0$.  This is essentially due to the S-shape of coherent reliability polynomials (see \cite{S-shape} for details).

Considering fixed points of node reliability, we note that there are node reliability polynomials with no fixed points in $(0,1)$ (for $n \geq 2$, $\rel(K_{1,n-1};p)=p + (n-1)p(1-p)^{n-1} > p$), node reliability polynomials with exactly one fixed point in $(0,1)$ (based on calculations for all graphs of small order it appears that the node reliability of any tree not isomorphic to a star has exactly one fixed point in this interval), and of course, exactly one node reliability polynomial (for the graph $K_1$) with all $p\in(0,1)$ being fixed points.  Surprisingly, there are many node reliability polynomials with two or more distinct fixed points in $(0,1),$ which again contrasts with what is known for coherent reliability polynomials.  We will prove that the node reliability of any sufficiently large $2$-connected graph of bounded degree has at least two fixed points in $(0,1).$  We will require the following lemma.

\begin{lemma}\label{RelDeriv}
If $G$ is a connected graph on $n$ vertices with $t$ cut vertices, then
\[
\rel'(G;0)=n
\]
and
\[
\rel'(G;1)=t
\]
\end{lemma}

\begin{proof}
Substituting into the expression for $\rel'(G;p)$ given in (\ref{RelDerivForm}) and using Observation \ref{ConnectedBasic} yields
\[
\rel'(G;0)=c_1=n
\] 
and 
\[
\rel'(G;1)=nc_n-c_{n-1}=n-(n-t)=t.\qedhere
\]
\end{proof}

The following result follows almost immediately from Lemma \ref{RelDeriv}.

\begin{corollary}
Let $G$ be a graph on $n$ vertices having $t\geq 2$ cut vertices.  Then $\rel(G;p)$ has at least one fixed point in $(0,1).$
\end{corollary}

\begin{proof}
By Lemma \ref{RelDeriv},
\[
\rel'(G;0)=n\geq 2 \ \ \ \ \mbox{and} \ \ \ \ \rel'(G;1)=t\geq 2.
\]
Therefore $\rel(G;p)>p$ on some interval $(0,\varepsilon_0)$ and $\rel(G;p)<p$ on some interval $(1-\varepsilon_1,1).$  By the Intermediate Value Theorem, $\rel(G;p)=p$ for some $p\in(\varepsilon_0,1-\varepsilon_1).$
\end{proof}

We conjecture that the node reliability polynomial of any graph with at least $2$ cut vertices has exactly one fixed point in $(0,1).$  We now prove the main result of this section which demonstrates that there are infinitely many graphs whose node reliability polynomials each have at least two fixed points in $(0,1).$

\begin{theorem}
Fix $\Delta\geq 2.$  Let $G$ be a $2$-connected graph on $n$ vertices with maximum degree $\Delta.$  For $n$ sufficiently large, $\rel(G;p)$ has at least two fixed points in $(0,1).$
\end{theorem}

\begin{proof}
By Lemma \ref{RelDeriv},
\[
\rel'(G;0)=n\geq 2 \ \ \ \ \mbox{and} \ \ \ \ \rel'(G;1)=t=0.
\]
Therefore $\rel(G;p)>p$ on some interval $(0,\varepsilon_0)$ and $\rel(G;p)>p$ on some interval $(1-\varepsilon_1,1).$  Now it is sufficient to prove that $\rel(G;p)<p$ for some $p\in(0,1),$ as the conclusion will follow from the Intermediate Value Theorem.  

We claim that $\rel\left(G;\tfrac{1}{\Delta^2}\right)<\tfrac{1}{\Delta^2}$ for $n$ sufficiently large.  The node reliability polynomial of $G$ is given by
\[
\rel(G;p)=np(1-p)^{n-1}+mp^2(1-p)^{n-2}+\sum_{k=3}^nc_kp^k(1-p)^k,
\]
where $c_k$ is the number of connected sets of $G$ of order $k$ for each $k\in\{3,\hdots,n\}.$

By Observation \ref{ConnectedBasic},
\[
c_3\leq n\binom{\Delta}{2}
\]
and so by Lemma \ref{GeneralBound},
\[
c_k\leq \frac{c_3}{k-2}\binom{n-3}{k-3}\leq \frac{n}{n-2}\binom{\Delta}{2}\binom{n-2}{k-2}
\]
for each $k\geq 3.$  Thus we have for $p \in (0,1)$ that
\begin{align*}
\sum_{k=3}^nc_kp^k(1-p)^k&\leq \sum_{k=3}^n\frac{n}{n-2}\binom{\Delta}{2}\binom{n-2}{k-2}p^k(1-p)^{n-k}\\
&=\frac{n}{n-2}\binom{\Delta}{2}\sum_{k=3}^n\binom{n-2}{k-2}p^k(1-p)^{n-k}\\
&=\frac{n}{n-2}\binom{\Delta}{2}p^2\sum_{k=1}^{n-2}\binom{n-2}{k}p^k(1-p)^{n-k-2}\\
&=\frac{n}{n-2}\binom{\Delta}{2}p^2\left[1-(1-p)^{n-2}\right]\\
&<\frac{n}{n-2}\binom{\Delta}{2}p^{2}.
\end{align*}
Using this bound on $\displaystyle\sum_{k=3}^nc_kp^k(1-p)^k$ and the elementary bound $m\leq \frac{n\Delta}{2},$ we have
\begin{align*}
\rel(G;p)&<np(1-p)^{n-1}+\frac{n\Delta}{2}p^2(1-p)^{n-2}+\frac{n}{n-2}\binom{\Delta}{2}p^2\\
&=np(1-p)^{n-2}\left[(1-p)+\frac{\Delta}{2}p\right]+\frac{n}{n-2}\binom{\Delta}{2}p^2.
\end{align*}
Therefore, 
\begin{align*}
\rel\left(G;\tfrac{1}{\Delta^2}\right)&<\tfrac{n}{\Delta^2}\left(1-\tfrac{1}{\Delta^2}\right)^{n-2}\left(1-\tfrac{1}{\Delta^2}+\tfrac{1}{2\Delta}\right)+\tfrac{n}{n-2}\tbinom{\Delta}{2}\tfrac{1}{\Delta^4}\\
&=\tfrac{n}{\Delta^2}\left(1-\tfrac{1}{\Delta^2}\right)^{n-2}\left(1-\tfrac{1}{\Delta^2}+\tfrac{1}{2\Delta}\right)+\tfrac{1}{2}\left(\tfrac{n}{n-2}\right)\left(\tfrac{\Delta-1}{\Delta}\right)\tfrac{1}{\Delta^2}
\end{align*}
For $n\geq 2\Delta$ we have
\[
\left(\tfrac{n}{n-2}\right)\left(\tfrac{\Delta-1}{\Delta}\right)\leq 1,
\] 
which implies
\[
\rel\left(G;\tfrac{1}{\Delta^2}\right)<n\left(1-\tfrac{1}{\Delta^2}\right)^{n-2}\left(1-\tfrac{1}{\Delta^2}+\tfrac{1}{2\Delta}\right)\tfrac{1}{\Delta^2}+\tfrac{1}{2\Delta^2}
\]
for $n\geq 2\Delta.$  It is clear that for $n$ sufficiently large we will have
\begin{align}\label{DeltaInequality}
n\left(1-\tfrac{1}{\Delta^2}\right)^{n-2}\left(1-\tfrac{1}{\Delta^2}+\tfrac{1}{2\Delta}\right)\leq\tfrac{1}{2}
\end{align}
as 
\[
\lim_{n\rightarrow\infty}n\left(1-\tfrac{1}{\Delta^2}\right)^{n-2}=0.
\]
We conclude that for $n$ sufficiently large,
\[
\rel\left(G;\tfrac{1}{\Delta^2}\right)<\tfrac{1}{2\Delta^2}+\tfrac{1}{2\Delta^2}=\tfrac{1}{\Delta^2}
\]
as claimed, and so the node reliability polynomial of $G$ has at least two fixed points in $(0,1)$ for $n$ sufficiently large (we note that the inequality (\ref{DeltaInequality}) can be solved exactly in terms of the well known Lambert W function in order to determine just how large $n$ must be in terms of $\Delta$).
\end{proof}

\section{Conclusion}

What is striking about node reliability is that on the surface its definition is analogous to that of other well known forms of reliability (such as all-terminal, two-terminal and $K$-terminal), but its shape and analytic properties are very different in general. The frequent lack of monotonicity, the contrasting concavity near $0$, the frequency of points of inflection, and the multiplicity of fixed points all illustrate that node reliability is quite different from the other models of probabilistic robustness on graphs (or even coherent systems), and merits further attention. 

We have found many graphs of small order whose node reliabilities each have three points of inflection in $(0,1).$  The graph shown in Figure \ref{3InflectionFig} is the unique graph on at most $7$ vertices satisfying this property.  We have found that the node reliabilities of $84$ of the $11117$ nonisomorphic connected graphs on $8$ vertices have three points of inflection in $(0,1).$  We note that all of the small graphs that we have found whose node reliabilities have three points of inflection have exactly one leaf and exactly one cut vertex.  Two questions arise:  Are there infinitely many graphs whose node reliabilities have $3$ inflection points in $(0,1)$?  Can the node reliability have arbitrarily many inflection points in $(0,1)$?

\begin{figure}[h]
\centering
\begin{tikzpicture}
\vertex (1) at (0,0) {};
\vertex (2) at (0,1) {};
\vertex (3) at (0,-1) {};
\vertex (4) at (-2,0) {};
\vertex (5) at (1,0.5) {};
\vertex (6) at (1,-0.5) {};
\vertex (7) at (-1,0) {};
\path
(1) edge (5)
(1) edge (6)
(1) edge (7)
(2) edge (5)
(2) edge (7)
(3) edge (6)
(3) edge (7)
(4) edge (7)
(5) edge (6)
  ;
\end{tikzpicture}
\caption{The unique graph of order at most $7$ whose node reliability has three points of inflection in $(0,1).$}\label{3InflectionFig}
\end{figure}
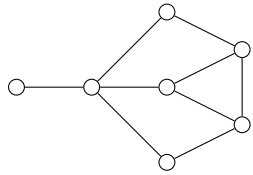

We conclude with yet another glaring difference between node reliability and all-terminal reliability. The all-terminal reliability polynomial of a disconnected graph is always identically zero, but the situation is not so trivial for the node reliability polynomial.  If $G$ is not connected, then
\[
\rel(G;0)=0 \ \ \ \ \mbox{and} \ \ \ \ \rel(G;1)=0,
\]
but $\rel(G;p)>0$ for all $p\in(0,1)$.  We have found many examples of disconnected graphs whose node reliability polynomials have two distinct maximal intervals of decrease in $(0,1)$ -- the plots of such functions are indeed `M-shaped', rather than the previously described `S-shaped' or `N-shaped'.  For example, the graph formed from the disjoint union of a single vertex and the star $K_{1,n-1}$ satisfies this property for all $n\geq 12.$  A plot of $\rel(K_{1,19}\cup K_1;p)$ is shown in Figure \ref{StarUnionVertex}. We ask: is there a connected graph whose node reliability has more than one maximal interval of decrease in $(0,1)$?

\begin{figure}[h!]
\begin{center}
\begin{overpic}[scale=0.3]{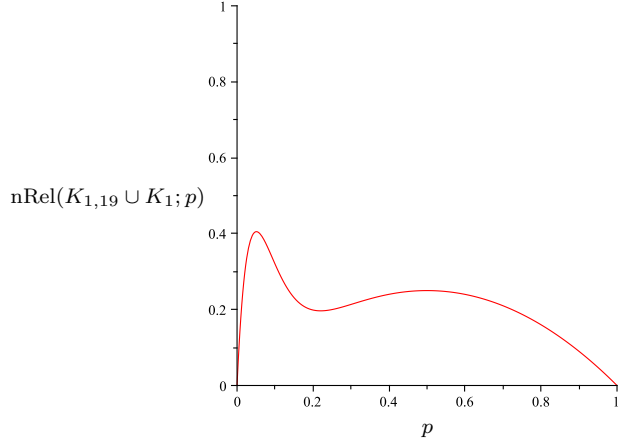}
\put(53,0){\footnotesize $p$}
\put(-40,52){\footnotesize $\rel(K_{1,19}\cup K_1;p)$}
\end{overpic}
\end{center}
\caption{The node reliability of the graph $K_{1,19}\cup K_1$.}
\label{StarUnionVertex}
\end{figure}

For those graphs whose node reliability polynomials have an interval of decrease in $(0,1),$ a natural question to ask is how long the interval of decrease can be.  For any $n\geq 2,$ the node reliability polynomial of the empty graph $\overline{K}_n$ on $n$ nodes (the complement of the complete graph $K_n$) is given by $\rel(\overline{K}_n;p)=np(1-p)^{n-1}$ which can easily be seen to be decreasing on the interval $(\tfrac{1}{n},1).$  This means that the interval of decrease can have length arbitrarily close to $1$ for disconnected graphs, but for connected graphs we conjecture that the length is at most $\tfrac{1}{2}.$  We can demonstrate that the length of the interval of decrease can be arbitrarily close to $\tfrac{1}{2}$ for connected graphs, and we give a brief sketch of this result.  Let $f_n$ be the node reliability polynomial of $K_{n-1}\circ K_2$ for each $n\geq 2.$  From the expression for $f^{\prime}_n$ given in (\ref{CompletePlusLeafDer}), we find that for $p \in \left(0,\tfrac{1}{2}\right)$,
\[ 
f^{\prime}_{n}(p) < g_{n}(p) = 2p-1+(n+1)(1-p)^{n-2}.
\]
We find that 
\[
\lim_{n \rightarrow \infty} \left\{g_{n}\left(\tfrac{1}{\ln n}\right)\right\} = -1  
\]
and
\[
\lim_{n \rightarrow \infty}  \left\{ng_{n}\left(\tfrac{1}{2}-\tfrac{1}{n}\right)\right\} = -2
\]
so that both $g_n\left(\tfrac{1}{\ln n}\right)<0$ and $g_n\left(\tfrac{1}{2}-\tfrac{1}{n}\right)<0$ for $n$ sufficiently large.
Also, 
\[
g^{\prime}_{n}(p) = 2-(n+1)(n-2)(1-p)^{n-3}
\] 
has a unique real root  
\[ 
q_{n}  = 1 - \left( \frac{2}{(n+1)(n-2)} \right)^{1/(n-3)} .
\]
As $\displaystyle\lim_{n \rightarrow \infty} \{q_n\ln n\} = 0$, it follows that for $n$ sufficiently large, $q_{n}$ is to the left of $\tfrac{1}{ \ln n}$, and so $g_{n}$, and hence $f^{\prime}_{n}$, is negative on $\left(\tfrac{1}{ \ln n}, \tfrac{1}{2}-\tfrac{1}{n}\right)$, which has length tending to $1/2$.

\section*{Acknowledgements}
 
The authors wish to thank the anonymous referees for their insightful comments and suggestions.  Research of Jason I.\ Brown is partially supported by grant RGPIN 170450-2013 from Natural Sciences and Engineering Research Council of Canada (NSERC).  Research of Lucas Mol is partially supported by an Alexander Graham Bell Canada Graduate Scholarship from NSERC.

\singlespacing

\bibliographystyle{amsplain}
\bibliography{NodeRelShape}

\end{document}